\documentclass{amsart}

\usepackage{amsmath}
\usepackage{amssymb}
\usepackage{graphicx}
\usepackage{tikz}
\usepackage{color}
\usepackage[all]{xy}
\usepackage{amsthm}
\usepackage{verbatim}
\usepackage{enumitem}

\usetikzlibrary{arrows,shapes,positioning}
\usetikzlibrary{decorations.markings}
\tikzstyle arrowstyle=[scale=1]
\tikzstyle directed=[postaction={decorate,decoration={markings,
    mark=at position .55 with {\arrow[arrowstyle]{stealth}}}}]
\tikzstyle ddirected=[postaction={decorate,decoration={markings,
    mark=at position .45 with {\arrow[arrowstyle]{stealth}},
    mark=at position .55 with {\arrow[arrowstyle]{stealth}}}}]
\tikzstyle reverse directed=[postaction={decorate,decoration={markings,
    mark=at position .45 with {\arrowreversed[arrowstyle]{stealth};}}}]
\tikzstyle reverse ddirected=[postaction={decorate,decoration={markings,
    mark=at position .45 with {\arrowreversed[arrowstyle]{stealth};},
    mark=at position .55 with {\arrowreversed[arrowstyle]{stealth};}}}]

\usepackage[colorlinks,linkcolor=blue,citecolor=blue,pdfstartview=FitH]{hyperref}

\theoremstyle{plain}

\newtheorem{para}{}[section]

\newtheorem{prop}[para]{Proposition}
\newtheorem{lemma}[para]{Lemma}

\newtheorem*{theorem}{Theorem}

\theoremstyle{remark}

\newtheorem{remark}{Remark}
\newtheorem*{remarkstar}{Remark}

\theoremstyle{definition}

\newcommand{\co}{\colon\thinspace}

\newcommand\bfx{\mathbf{x}}

\begin{document}

\title{Incompressible solvable representations of surface groups}
\author{Jason DeBlois}
\address{Department of Mathematics\\University of Pittsburgh\\301 Thackeray Hall\\Pittsburgh, PA 15260}
\email{jdeblois@pitt.edu}

\author{Daniel Gomez}
\email{dannielgomez87@gmail.com}

\begin{abstract}
The fundamental group of every surface that is not the projective plane or Klein bottle has a representation to a torsion-free group of upper-triangular matrices in $\mathrm{SL}_2(\mathbb{R})$ with no simple loop (i.e. a nontrivial element representing a simple closed curve) in the kernel.
\end{abstract}

\maketitle

We will call a representation (i.e.~a homomorphism) from the fundamental group $\pi_1 S$ of a surface $S$ to another group  \textit{incompressible} if its kernel contains no simple loop; i.e. a non-trivial element representing a simple closed curve in $S$.

\begin{theorem}  For any compact surface $\Sigma$, possibly with boundary, that is not the projective plane or Klein bottle there is an incompressible representation of $\pi_1 \Sigma$ to a torsion-free group of upper triangular matrices in $\mathrm{SL}_2(\mathbb{R})$, which therefore injects under the quotient map to $\mathrm{PSL}_2(\mathbb{K})$ or $\mathrm{PGL}_2(\mathbb{K})$, for $\mathbb{K} = \mathbb{R}$ or $\mathbb{C}$.\end{theorem}

\begin{remark}  It is not hard to show directly that the projective plane's and Klein bottle's fundamental groups have incompressible representations to upper-triangular matrices in $\mathrm{SL}_2(\mathbb{R})$, but these all have torsion and do not remain incompressible after projectivizing.  See Remark \ref{Klein}.\end{remark}

\begin{remark} The theorem above implies in particular that every simple loop in $\pi_1\Sigma$ survives its \textit{metabelianization}, the quotient by the second term of its derived series.  As a sort of sanity check, we note that this was shown directly in \cite[Example 18]{Zemke}.\end{remark}

This implies an(other) answer to a question of Yair Minsky.  Motivated by the three-dimensional simple loop conjecture, he asked whether there exist non-injective but incompressible representations of hyperbolic surface groups into $\mathrm{PSL}_2(\mathbb{C})$ \cite[Question 5.3]{Minsky}.  Minsky's question has been answered in the affirmative, rather emphatically at this point, in independent works of Cooper--Manning \cite{CoopMan}, Louder \cite{Louder}, Danny Calegari \cite{Calegari}, and Mann \cite{Mann}.  These interesting papers each provide non-injective, incompressible representations that are customized in different ways.

We think this result is worth adding to the pile for a few reasons.  First, our representations are as simple as possible, both in terms of the target's dimension (two) and the fact that it is two-step solvable --- any one-step solvable (i.e.~abelian) representation kills all nullhomologous curves.  And their existence is not implied by previous work.  The set of upper triangular representations in $\mathrm{SL}_2(\mathbb{R})$ has positive codimension in the variety considered by Cooper--Manning; Louder's representations factor through embeddings of limit groups; Mann's are faithful on certain free subgroups of $\pi_1\Sigma$; and Calegari certifies incompressibility using positivity of \textit{scl}, which vanishes on the commutator subgroup of a solvable group (cf.~\cite{scl}).

Our proof also represents a particular approach to proving existence of representations with nice properties, boiled down to its bare bones.  We show for any fixed simple closed curve that the set of upper-triangular representations that kill it is closed and nowhere dense in the set of all such representations, equipped with its  topology inherited from Euclidean space.  The set of incompressible upper-triangular representations is thus a countable intersection of open dense subsets, so by ``Baire's Theorem'' (see eg.~\cite[Theorem 7.2]{Munkres}) it is non-empty.

Cooper--Manning use the same approach, with a different set of representations.  But the spaces and equations that we encounter here are so simple that we need no algebro-geometric machinery.  And our representation spaces ``treat all curves equally'', allowing a key simplification: we can (and do) check nowhere denseness only on a very small set of representatives for mapping class group orbits of simple closed curves.  Our argument is close to self-contained, appealing only to the classification of surfaces, some very basic facts about $\pi_1$, and Baire's theorem.

Finally, our result completes the statement of Theorem 1.3 of \cite{Mann}.  Mann, the only previous author to consider the non-orientable case, proved there that the fundamental group of every non-orientable hyperbolic surface admits non-faithful incompressible representations to $\mathrm{PGL}_2(\mathbb{R})$, save possibly those of the punctured Klein bottle or closed non-orientable genus-$3$ surface.  We cover these cases.  

In fact, although our general method does not provide explicit representations, in the spirit of \cite{Mann} we will describe some for the non-orientable genus-$3$ surface $\Upsilon_3$.  It is well known that the punctured Klein bottle $K$ is a $\pi_1$-injective subsurface of $\Upsilon_3$, so these also give incompressible representations of $\pi_1 K$.  Alternatively, such representations are exhibited directly in \cite[\S 4]{Gomez}.

\newcommand\GenusThreeProp{For the fundamental group of the non-orientable genus-$3$ surface, presented as
\[ \langle a,b,c\,|\, aba^{-1}b^{-1}c^2 = 1 \rangle, \]
a non-injective, incompressible representation to a torsion--free subgroup of $\mathrm{SL}_2(\mathbb{R})$ is determined by $a\mapsto A$, $b\mapsto B$ and $c\mapsto C$, where\begin{align*}
	& A= \begin{pmatrix} x & y \\ 0 & 1/x \end{pmatrix} &
	& B= \begin{pmatrix} z & w \\ 0 & 1/z \end{pmatrix} &
	& C= \begin{pmatrix} 1 & \frac{-1}{2}\left(xy(1-z^2)-zw(1-x^2)\right) \\ 0 & 1\end{pmatrix}
\end{align*}
for any $x,z \in(0,\infty)$ such that $x^pz^q\neq 1$ for all $(p,q)\in\mathbb{Z}^2-\{(0,0)\}$, and any $y,w\in\mathbb{R}$ such that $p(x,y,z,w)\doteq xy(1-z^2)-zw(1-x^2)\ne 0$.  Hence this determines a representation to $\mathrm{PGL}_2(\mathbb{R})$ satisfying (1) through (3) of \cite[Th.~1.2]{Mann}.

Changing the lower-right entry of $C$ to $-1$ produces an incompressible, upper-triangular representation to $\mathrm{PGL}_2(\mathbb{R})$ that is \mbox{\rm two-sided} in the sense of Sections 2 and 6 of \cite{Zemke}, when $\mathrm{PGL}_2(\mathbb{R})$ is given the orientation character that measures the sign of the determinant.}
\theoremstyle{plain}
\newtheorem*{genus3prop}{Proposition \ref{genus 3}}
\begin{genus3prop}\GenusThreeProp\end{genus3prop}

These representations need not be very exotic: $(x,y,z,w)=(2,1,3,0)$ works, for instance.  Studying this case was what originally led us to consider upper-triangular representations to $\mathrm{SL}_2(\mathbb{R})$.  We prove Proposition \ref{genus 3} by simply checking that no simple loop lies in the kernel of the given representation, using the complex of curves of the non-orientable genus-$3$ surface as described by Scharlemann \cite{Scharlemann}.

\subsection*{Acknowledgements}  The first author thanks Ben McReynolds and Alan Reid for helpful correspondence.  We are grateful to the referee for kind and helpful remarks.

\section{The orientable case}\label{orientable}

We address the orientable case of the main theorem first, since it exhibits the proof's main features while being more straightforward in terms of the number of cases to check.  In Section \ref{groups} we fix a description of the fundamental group of an orientable surface and give an explicit classification, up to automorphism, of elements representing simple closed curves.  We topologize the set of upper-triangular representations to $\mathrm{SL}_2(\mathbb{R})$ in Section \ref{spaces} and complete the proof in Section \ref{orientable proof}.

\subsection{Fundamental groups}\label{groups}  For $g\geq 1$, let $\Sigma_g$ be the closed, orientable surface of genus $g$.  We will use the following presentation for its fundamental group $\pi_1\Sigma_g$:
\begin{align}\label{closed or} \pi_1 \Sigma_g \cong \langle a_1,b_1,\hdots,a_g,b_g\,|\,[a_1,b_1]\cdots[a_g,b_g] = 1\rangle \end{align}
Here $[a,b] = aba^{-1}b^{-1}$ is the \textit{commutator} of $a$ and $b$.  From the standard description of $\Sigma_g$ as an identification space of a $4g$-gon it is easy to see that each generator represents a non-separating simple closed curve, and for each $g_0<g$, that the product of commutators $[a_1,b_1]\cdots[a_{g_0},b_{g_0}]$ represents a separating simple closed curve $\gamma_{g_0}$. 

Moreover the surface obtained by \textit{cutting} $\Sigma_g$ along $\gamma_{g_0}$ (see \cite[\S 1.3.1]{FaMa}), which can be taken as the complement in $\Sigma_g$ of a small open regular neighborhood $\mathcal{N}(\gamma_{g_0})$ of $\gamma_{g_0}$, has one component that is a genus-$g_0$ subsurface and one with genus $g-g_0$, each with one boundary component.  We have:

\begin{lemma}\label{oogedy boogedy} Let $\gamma$ be a non-nullhomotopic simple closed curve on $\Sigma_g$.  If $\gamma$ is non-separating then there is an automorphism $\phi$ of $\Sigma_g$ such that $\phi(\gamma)$ is represented in $\pi_1\Sigma_g$ by $a_1$. If $\gamma$ is separating then there is an automorphism $\phi$ of $\Sigma_g$ such that $\phi(\gamma)$ is represented in $\pi_1\Sigma_g$ by $[a_1,b_1]\cdots[a_{g_0},b_{g_0}]$ for some $g_0<g$.\end{lemma}

This standard fact is an exercise in applying the classification of surfaces (see eg. \cite[Theorem 1.1]{FaMa}) together with the ``change of coordinates principle'' discussed at length in Section 1.3 of the Primer on Mapping Class Groups \cite{FaMa} (thanks to the referee for this reference).  We leave it to the reader but refer also to Lemma \ref{six cases}, where we give a proof of the analogous classification for non-orientable surfaces.



\subsection{Representation spaces}\label{spaces}  For matrices $A = \left(\begin{smallmatrix} x & y \\ 0 & 1/x \end{smallmatrix}\right)$ and $B = \left(\begin{smallmatrix} z & w \\ 0 & 1/z \end{smallmatrix}\right)$, where $x$, $y$, $z$ and $w$ are real with $x,z\neq 0$, a simple direct computation gives\begin{align*}
	[A,B] = \left(\begin{smallmatrix} x & y \\ 0 & 1/x \end{smallmatrix}\right)\left(\begin{smallmatrix} z & w \\ 0 & 1/z \end{smallmatrix}\right)\left(\begin{smallmatrix} 1/x & -y \\ 0 & x \end{smallmatrix}\right)\left(\begin{smallmatrix} 1/z & -w \\ 0 & z \end{smallmatrix}\right) = \left(\begin{smallmatrix} 1 & xy(1-z^2) - zw(1-x^2) \\ 0 & 1 \end{smallmatrix}\right) \end{align*}
Given this, we take $p(x,y,z,w) = xy(1-z^2) - zw(1-x^2)$ and for the closed, orientable surface $\Sigma_g$ of genus $g$ define the \textit{space of upper triangular representations} from $\pi_1\Sigma_g$ to $\mathrm{SL}_2(\mathbb{R})$ as:\begin{align*}
	\mathfrak{U}(g) & = \{\ (x_1,y_1,z_1,w_1,\hdots,x_g,y_g,z_g,w_g)\ |\ \sum_{i=1}^g p(x_i,y_i,z_i,w_i) =0, \\
	   &\ \hspace{1.3in}\mbox{where}\ x_i,y_i,z_i,w_i\in\mathbb{R},\ \mbox{with}\ x_i,z_i\neq 0,\ \mbox{for all}\ i\ \}, \end{align*}
topologized as a subspace of $\mathbb{R}^{4g}$.  A point of $\mathfrak{U}(g)$ determines a representation of $\pi_1(\Sigma_g)$, as presented in (\ref{closed or}), by sending $a_i$ to $A_i = \left(\begin{smallmatrix} x_i & y_i \\ 0 & 1/x_i \end{smallmatrix}\right)$ and $b_i$ to $B_i = \left(\begin{smallmatrix} z_i & w_i \\ 0 & 1/z_i \end{smallmatrix}\right)$ for each $i$.  This evidently gives a bijective correspondence between $\mathfrak{U}(g)$ and the set of upper-triangular representations $\pi_1 \Sigma_g\to \mathrm{SL}_2(\mathbb{R})$.

The map $(x_1,\hdots,w_g)\mapsto(A_1,\hdots,B_g)$ defined above embeds $\mathfrak{U}(g)$ into $M_2(\mathbb{R})^{2g}$, where the set $M_2(\mathbb{R})$ of two-by-two real matrices is itself identified with $\mathbb{R}^4$.  We note that the image of this embedding is closed.  In particular, the upper-triangular matrices $\left\{\left(\begin{smallmatrix} x & y \\ 0 & 1/x\end{smallmatrix}\right)\right\}$ in $\mathrm{SL}_2(\mathbb{R})$ comprise a closed subset of $M_2(\mathbb{R})$, even though $x$ is constrained to lie in $\mathbb{R}-\{0\}$, for the same reason that the graph $\{(x,1/x)\,|\,x\neq0\}$ is a closed subset of $\mathbb{R}^2$. 

Regarding $\mathfrak{U}(g)$ as embedded in $M_2(\mathbb{R})^{2g}$ we obtain the following standard fact.

\begin{lemma}\label{pooka pooka} For any $g>0$ and automorphism $\phi$ of $\pi_1\Sigma_g$, there is a homeomorphism $\Phi$ of $\mathfrak{U}(g)$ taking $\rho$ to $\rho\circ\phi$ for each representation $\rho\co\pi_1 \Sigma_g\to\mathrm{SL}_2(\mathbb{R})$.\end{lemma}

\begin{proof}  For each $i$, write $\phi(a_i)=v_i(a_1,\hdots,b_g)$ and $\phi(b_i)=w_i(a_1,\hdots,b_g)$ for \textit{words} $v_i$ and $w_i$, that is, finite products of the generators and their inverses.  For any upper-triangular representation $\rho\co\pi_1(\Sigma_g)\to\mathrm{SL}_2(\mathbb{R})$, if $\rho(a_i) = A_i$ and $\rho(b_i)= B_i$ then $\rho\circ\phi (a_i) = v_i(A_1,\hdots,B_g)$ and $\rho\circ\phi (b_i) = w_i(A_1,\hdots,B_g)$ for each $i$.  So we define $\Phi\co(M_2(\mathbb{R}))^{2g}\to(M_2(\mathbb{R}))^{2g}$ componentwise by setting $\Phi_i=v_i$ for $i$ odd and $\Phi_i=w_i$ for $i$ even.  Here we realize inversion on $\mathrm{SL}_2(\mathbb{R})$ by the map on $M_2(\mathbb{R})$ that swaps diagonal entries and multiplies off-diagonal entries by $-1$.  $\Phi$ is thus continuous since this map and matrix multiplication are given by polynomial operations on the entries of two-by-two matrices.

That $\Phi$ is a homeomorphism follows from the additional observation that on $(\mathrm{SL}_2(\mathbb{R}))^{2g}$ the assignment $\phi\to\Phi$ is (contravariant) functorial: it takes the identity to the identity and $\phi_1\circ\phi_2$ to $\Phi_2\circ\Phi_1$.  \end{proof}

\subsection{Proof of the main theorem}\label{orientable proof} The with-boundary case follows from the closed case, since every compact surface with boundary is $\pi_1$-injective in the surface obtained by doubling it across its boundary, being the image of a retraction of the double.  So we assume henceforth that $\Sigma$ is closed.  

If $\Sigma$ is the sphere then the trivial representation of $\pi_1\Sigma = \{1\}$ is vacuously incompressible, and if $\Sigma = \Sigma_1$ is the torus then one defines a faithful (hence incompressible) upper-triangular representation of $\pi_1\Sigma \cong \langle a,b\,|\,[a,b] = 1\rangle\cong\mathbb{Z}^2$ by sending $a$ to $\left(\begin{smallmatrix} x & 0 \\ 0 & 1/x \end{smallmatrix}\right)$ and $b$ to $\left(\begin{smallmatrix} z & 0 \\ 0 & 1/z \end{smallmatrix}\right)$ for rationally independent real numbers $x$ and $z$.  So we will assume below that $\Sigma = \Sigma_g$ for $g\geq 2$.

As we observed in Section \ref{spaces}, $\mathfrak{U}(g)$ is the intersection in $M_2(\mathbb{R})^{2g}$ of the $2g$-fold Cartesian product of upper-triangular matrices in $\mathrm{SL}_2(\mathbb{R})$ with the zero set of 
\[ p_g\left(\left(\begin{smallmatrix} x_1 & y_1 \\ * & * \end{smallmatrix}\right),
	\left(\begin{smallmatrix} z_1 & w_1 \\ * & * \end{smallmatrix}\right),\hdots,
	\left(\begin{smallmatrix} z_g & w_g \\ * & * \end{smallmatrix}\right)\right) = \sum_{i=1}^g p(x_i,y_i,z_i,w_i) , \]
where $p(x,y,z,w) = xy(1-z^2) - zw(1-x^2)$.  It is therefore closed in $M_2(\mathbb{R})^{2g}$, hence the subspace metric it inherits from the Euclidean metric on $M_2(\mathbb{R})^{2g} = \mathbb{R}^{8g}$ is complete.  For any non-nullhomotopic simple closed curve $\gamma$ on $\Sigma$, we will show that the set of upper-triangular representations $\pi_1 S\to\mathrm{SL}_2(\mathbb{R})$ in $\mathfrak{U}(g)$ that have an element representing $\gamma$ in the kernel is a closed, nowhere-dense subset of $\mathfrak{U}(g)$.  Baire's theorem (see eg.~\cite[Theorem 7.2]{Munkres}) will then imply that the set of incompressible upper-triangular representations into $\mathrm{SL}_2(\mathbb{R})$ is non-empty, being a countable intersection of open dense subsets of $\mathfrak{U}(g)$.

First suppose that $\gamma$ is non-separating.  Then by Lemma \ref{oogedy boogedy} there is an automorphism $\phi$ of $\Sigma$ that sends $\gamma$ to a curve represented by $a_1$.  Upon fixing a basepoint $p\in\Sigma$ and an arc joining $p$ to its image under $\phi$ we obtain an induced automorphism $\phi_*$ of $\pi_1(\Sigma,p)$.  If $[\gamma]\in\pi_1(\Sigma,p)$ is a based homotopy class representing $\gamma$ then $\phi_*([\gamma])$ is conjugate to $a_1$, hence it lies in the kernel of a representation $\rho$ if and only if $a_1$ does.  Since $\rho\circ\phi_* = \Phi(\rho)$ for each representation $\rho$, where $\Phi$ is the homeomorphism of $\mathfrak{U}(g)$ supplied by Lemma \ref{pooka pooka}, the set of representations that kill a representative of $\gamma$ is the image under $\Phi$ of those that kill $a_1$.  In particular, it is closed and nowhere-dense in $\mathfrak{U}(g)$ if and only if this holds for those that kill $a_1$.

The set of representations that kill $a_1$ is certainly closed in $\mathfrak{U}(g)\subset\mathbb{R}^{4g}$, being its intersection with $\{(1,0)\}\times\mathbb{R}^{4g-2}$.  For any such representation $(1,0,z_1,\hdots,w_g)$, $p(1,0,z_1,w_1) = 0$.  If $z_1= \pm 1$ then for $\epsilon\neq0$, one easily checks that the representations $(1,\epsilon,z_1,\hdots,w_g)$ lie in $\mathfrak{U}(g)$, limit to $(1,0,z_1,\hdots,w_g)$ as $\epsilon\to 0$, and map $a_1$ non-trivially.  If $z_1\ne\pm 1$ then for $\epsilon>0$ one may take $x_1(\epsilon) = 1+\epsilon$ and solve the equation $p(x_1(\epsilon),y_1(\epsilon),z_1,w_1)=0$ for $y_1(\epsilon)$, yielding
\[ y_1(\epsilon) = -\epsilon\,\frac{2+\epsilon}{1+\epsilon}\,\frac{z_1w_1}{1-z_1^2} \]
The representations $(1+\epsilon,y_1(\epsilon),z_1,\hdots,w_g)$ thus lie in $\mathfrak{U}(g)$, map $a_1$ non-trivially, and limit to $(1,0,z_1,\hdots,w_g)$ as $\epsilon\to0$.  It follows that the set of representations that kill $a_1$ is nowhere dense in $\mathfrak{U}(g)$.

\begin{remarkstar} As the referee has observed, the argument above simply shows that $\pm\left(\begin{smallmatrix} 1 & 0 \\ 0 & 1 \end{smallmatrix}\right)$ is not an isolated point of the centralizer of $\left(\begin{smallmatrix} z_1 & w_1 \\ 0 & 1/z_1 \end{smallmatrix}\right)$ among upper-triangular matrices in $\mathrm{SL}_2(\mathbb{R})$. This is fairly obvious from the standpoint of hyperbolic geometry.\end{remarkstar}

If $\gamma$ is separating then an analogous argument appealing to Lemmas \ref{oogedy boogedy} and \ref{pooka pooka} shows that the set of upper-triangular representations to $\mathrm{SL}_2(\mathbb{R})$ that kill a representative of $\gamma$ is closed and nowhere dense if and only if those that kill $[a_1,b_1]\cdots[a_{g_0},b_{g_0}]$ is, for some $g_0$ with $1\leq g_0<g$.  To show this, it will help to know that the ``generic'' zero of the polynomial $p$ has a certain form.

\begin{lemma}\label{pee}  For $p(x,y,z,w) = xy(1-z^2) - zw(1-x^2)$, if $p(x,y,z,w) = 0$ for some $(x,y,z,w)$ with $x,z\neq 0$, then there is a sequence $(x_n,y_n,z_n,w_n)$ converging to $(x,y,z,w)$ as $n\to\infty$ with $p(x_n,y_n,z_n,w_n) = 0$ and either $x_n \neq \pm1$ for all $n$ or $z_n\neq \pm1$ for all $n$.\end{lemma}

\begin{proof}  We may as well assume $x,z\in\{\pm 1\}$ since otherwise we can use the constant sequence.  Assuming first that $y\neq 0$, we choose a sequence $(x_n)$ converging to $x$ as $n\to\infty$, with $x_n\neq \pm 1$ for all $n$, and with $y$ and $w$ fixed, determine $z_n$ by
\[ z_n = \frac{w(1-x_n^2)\pm \sqrt{w^2(1-x_n^2)^2+4x_n^2y^2}}{-2x_ny} \]
For fixed $y\ne0$ and $w$, and a fixed choice of sign for the radical above, this converges to $\pm 1$ as $x_n\to x$, where the sign of the limit depends only on the signs of $x$ and $y$ and the sign choice for the radical. So we consistently choose the sign of the radical according to what gives the correct value $z\in\{\pm 1\}$ at $x$. 

If $y=0$ then taking $x_n = x\in\{\pm 1\}$ and $y_n = y$ for all $y$, $p(x_n,y_n,z_n,w_n)=0$ for \textit{any} choice of $z_n$ and $w_n$; in particular, for one with $z_n\neq \pm1$ for all $n$ and $z_n\to z$ and $w_n\to w$ as $n\to\infty$.  (This reflects that $\pm\left(\begin{smallmatrix} 1 & 0 \\ 0 & 1 \end{smallmatrix}\right)$ is central in $\mathrm{SL}_2(\mathbb{R})$.)\end{proof}

We now show that for $1\leq g_0<g$, the intersection of $\mathfrak{U}(g)$ with the zero set of $p_{g_0}$ is nowhere dense in $\mathfrak{U}(g)$, where $p_{g_0}$ is regarded as a function on $M_2(\mathbb{R})^{2g}$ by composing with projection to the first $2g_0$ factors.  A point $\bfx=(x_1,\hdots,w_g)$ of $\mathfrak{U}(g)$ in the zero set of $p_{g_0}$ satisfies both
\[ \sum_{i=1}^{g_0} p(x_i,y_i,z_i,w_i) = 0 \quad \mbox{and}\quad \sum_{i=g_0+1}^g p(x_i,y_i,z_i,w_i) = 0, \]
the latter sum being $p_g(\bfx) - p_{g_0}(\bfx)$.  Applying Lemma \ref{pee}, we may perturb $\bfx$ by an arbitrarily small amount and thus ensure that $x_{i}\neq \pm 1$ or $z_{i}\neq\pm 1$ for some $i\leq g_0$ and also $x_{j}\neq\pm 1$ or $z_{j}\neq \pm1$ for some $j\geq g_0+1$.  (Note that this holds automatically for any $i$ such that $p(x_i,y_i,z_i,w_i)\neq 0$.)

Supposing that $x_{i}\neq\pm 1$ and $x_{j}\neq\pm 1$ (the other cases are similar), for any $\epsilon>0$ we may replace $w_i$ by $w_i+\frac{\epsilon}{z_i(1-x_{i}^2)}$ and $w_j$ by $w_j-\frac{\epsilon}{z_j(1-x_j^2)}$ to produce $\bfx_{\epsilon}\in\mathfrak{U}(g)$ with $p_{g_0}(\bfx_{\epsilon}) =-\epsilon$.  There is thus a neighborhood of $\bfx_{\epsilon}$ where $p_{g_0}$ is non-zero, showing that its zero set is nowhere dense as asserted.

As promised, Baire's theorem now implies that there is an incompressible representation from $\pi_1\Sigma$ to an upper-triangular subgroup of $\mathrm{SL}_2(\mathbb{R})$.  One can easily show directly that the only upper-triangular torsion element of $\mathrm{SL}_2(\mathbb{R})$ is $\left(\begin{smallmatrix} -1 & 0 \\ 0 & -1\end{smallmatrix}\right)$.  And $\mathfrak{U}(g)$ has at least $2^{2g}$ connected components determined by the signs of the $x_i$ and $z_i$ (which, we recall, cannot vanish).  We may therefore produce an incompressible representation to a torsion-free subgroup by simply running the argument above in the component containing $(1,0)^{2g}$, where all $x_i$ and $z_i$ are positive.

We finally note that the projectivization maps $\mathrm{SL}_2(\mathbb{K})\to\mathrm{PSL}_2(\mathbb{K})$ and $\mathrm{GL}_2(\mathbb{K})\to\mathrm{PGL}_2(\mathbb{K})$ have kernels consisting of diagonal matrices, for both $\mathbb{K}=\mathbb{R}$ and $\mathbb{K}=\mathbb{C}$.  Since the only non-trivial diagonal element of $\mathrm{SL}_2(\mathbb{R})$ is $\left(\begin{smallmatrix} -1 & 0 \\ 0 & -1 \end{smallmatrix}\right)$, no torsion-free subgroup intersects the kernel of a projectivization map.

\section{The non-orientable case}

The classification of simple closed curves up to automorphism is somewhat more complicated on non-orientable than on orientable surfaces.  To accommodate this we will use three different presentations, based on the polygonal decompositions of Figure \ref{non-orientable models}, for the fundamental group of a non-orientable surface $\Upsilon_n$ of genus $n$, by which we mean the connected sum of $n$ copies of the projective plane.  (Here we will always use ``$n$'' to refer to the non-orientable genus, versus ``$g$'' which continues to reference the orientable genus.)

In the figure, for $k\in\mathbb{N}$, $P_k$ refers to an equilateral Euclidean $k$-gon, and its edges are identified in pairs according to the labeling by homeomorphisms that match the orientations specified by the arrows.  One easily checks in each case that the quotient cell complex has one vertex and is homeomorphic to a non-orientable surface.  In particular, a simple closed curve that intersects the quotient of any edge pair labeled with a ``$c$'' exactly once has a M\"obius band neighborhood.  The red arcs in the figure quotient to curves with this property.

\begin{figure}
\begin{tikzpicture}[scale=0.5]

\begin{scope}
    \fill [opacity=0.1] (4.1,-1.1) arc (30:180:2.2);
    \draw [dotted] (4.1,-1.1) arc (30:180:2.2);
    \fill [opacity=0.1] (0,-2.2) -- (0,-2.8) -- (0.6,-3.8) -- (1.6,-4.4) -- (2.8,-4.4) -- (3.8,-3.8) -- (4.4,-2.8) -- (4.4,-1.6) -- (4.1,-1.1);
    \draw (0,-2.2) -- (0,-2.8);
    \draw (4.4,-1.6) -- (4.1,-1.1);
    \draw [red] (1.1,-4.1) arc (-30:120:0.6);
    \draw [directed] (0,-2.8) -- (0.6,-3.8);
    \draw [directed] (0.6,-3.8) -- (1.6,-4.4);
    \draw [directed] (1.6,-4.4) -- (2.8,-4.4);
    \draw [directed] (2.8,-4.4) -- (3.8,-3.8);
    \draw [directed] (3.8,-3.8) -- (4.4,-2.8);
    \draw [directed] (4.4,-2.8) -- (4.4,-1.6);
    \draw [fill] (0,-2.8) circle [radius=0.1];
    \draw [fill] (1.6,-4.4) circle [radius=0.1];
    \draw [fill] (3.8,-3.8) circle [radius=0.1];
    \draw [fill] (0.6,-3.8) circle [radius=0.1];
    \draw [fill] (2.8,-4.4) circle [radius=0.1];
    \draw [fill] (4.4,-2.8) circle [radius=0.1];
    \draw [fill] (4.4,-1.6) circle [radius=0.1];
    
    \node [below left] at (0.3,-3.3) {$c_n$};
    \node [below left] at (1.3,-4.1) {$c_n$};
    \node [below] at (2.2,-4.4) {$c_1$};
    \node [below right] at (3.1,-4.1) {$c_1$};
    \node [below right] at (4.1,-3.3) {$c_2$};
    \node [right] at (4.4,-2.2) {$c_2$};
    
    \node at (2.2,-2.2) {\large $P_{2n}$};
    \node at (0,0) {(2)};
\end{scope}

\begin{scope}[xshift=8cm]
    \fill [opacity=0.1] (4.1,-1.1) arc (30:150:2.2);
    \draw [dotted] (4.1,-1.1) arc (30:150:2.2);
    \fill [opacity=0.1] (0.3,-1.1) -- (0,-1.6) -- (0,-2.8) -- (0.6,-3.8) -- (1.6,-4.4) -- (2.8,-4.4) -- (3.8,-3.8) -- (4.4,-2.8) -- (4.4,-1.6) -- (4.1,-1.1);
    \draw (0.3,-1.1) -- (0,-1.6);
    \draw (4.4,-1.6) -- (4.1,-1.1);
    \draw [red] (1.1,-4.1) arc (-30:120:0.6);
    \draw [reverse directed] (0,-1.6) -- (0,-2.8);
    \draw [directed] (0,-2.8) -- (0.6,-3.8);
    \draw [directed] (0.6,-3.8) -- (1.6,-4.4);
    \draw [directed] (1.6,-4.4) -- (2.8,-4.4);
    \draw [directed] (2.8,-4.4) -- (3.8,-3.8);
    \draw [reverse directed] (3.8,-3.8) -- (4.4,-2.8);
    \draw [reverse directed] (4.4,-2.8) -- (4.4,-1.6);
    \draw [fill] (0,-1.6) circle [radius=0.1];
    \draw [fill] (0,-2.8) circle [radius=0.1];
    \draw [fill] (1.6,-4.4) circle [radius=0.1];
    \draw [fill] (3.8,-3.8) circle [radius=0.1];
    \draw [fill] (0.6,-3.8) circle [radius=0.1];
    \draw [fill] (2.8,-4.4) circle [radius=0.1];
    \draw [fill] (4.4,-2.8) circle [radius=0.1];
    \draw [fill] (4.4,-1.6) circle [radius=0.1];
    
    \node [left] at (0,-2.2) {$b_g$};
    \node [below left] at (0.3,-3.3) {$c$};
    \node [below left] at (1.3,-4.1) {$c$};
    \node [below] at (2.2,-4.4) {$a_1$};
    \node [below right] at (3.1,-4.1) {$b_1$};
    \node [below right] at (4.1,-3.3) {$a_1$};
    \node [right] at (4.4,-2.2) {$b_1$};
    
    \node at (2.2,-2.2) {\large $P_{4g+2}$};
    \node at (0,0) {(3)};
\end{scope}

\begin{scope}[xshift=16cm]
     \fill [opacity=0.1] (4.1,-1.1) arc (30:120:2.2);
    \draw [dotted] (4.1,-1.1) arc (30:120:2.2);
    \fill [opacity=0.1] (1.1,-0.3) -- (0.6,-0.6) -- (0,-1.6) -- (0,-2.8) -- (0.6,-3.8) -- (1.6,-4.4) -- (2.8,-4.4) -- (3.8,-3.8) -- (4.4,-2.8) -- (4.4,-1.6) -- (4.1,-1.1);
    \draw (1.1,-0.3) -- (0.6,-0.6);
    \draw (4.4,-1.6) -- (4.1,-1.1);
    \draw [red] (0.3,-3.3) arc (-74.745:74.745:1.14);
    \draw [directed] (0.6,-0.6) -- (0,-1.6);
    \draw [directed] (0,-1.6) -- (0,-2.8);
    \draw [directed] (0,-2.8) -- (0.6,-3.8);
    \draw [reverse directed] (0.6,-3.8) -- (1.6,-4.4);
    \draw [directed] (1.6,-4.4) -- (2.8,-4.4);
    \draw [directed] (2.8,-4.4) -- (3.8,-3.8);
    \draw [reverse directed] (3.8,-3.8) -- (4.4,-2.8);
    \draw [reverse directed] (4.4,-2.8) -- (4.4,-1.6);
    \draw [fill] (0.6,-0.6) circle [radius=0.1];
    \draw [fill] (0,-1.6) circle [radius=0.1];
    \draw [fill] (0,-2.8) circle [radius=0.1];
    \draw [fill] (1.6,-4.4) circle [radius=0.1];
    \draw [fill] (3.8,-3.8) circle [radius=0.1];
    \draw [fill] (0.6,-3.8) circle [radius=0.1];
    \draw [fill] (2.8,-4.4) circle [radius=0.1];
    \draw [fill] (4.4,-2.8) circle [radius=0.1];
    \draw [fill] (4.4,-1.6) circle [radius=0.1];
    
    \node [above left] at (0.3,-1.2) {$c$};
    \node [left] at (0,-2.2) {$d$};
    \node [below left] at (0.3,-3.3) {$c$};
    \node [below left] at (1.3,-4.1) {$d$};
    \node [below] at (2.2,-4.4) {$a_1$};
    \node [below right] at (3.1,-4.1) {$b_1$};
    \node [below right] at (4.1,-3.3) {$a_1$};
    \node [right] at (4.4,-2.2) {$b_1$};
    
    \node at (2.2,-2.2) {\large $P_{4g+4}$};
    \node at (0,0) {(4)};
\end{scope}

\end{tikzpicture}
\caption{Polygonal edge pairings yielding non-orientable surfaces.}
\label{non-orientable models}
\end{figure}
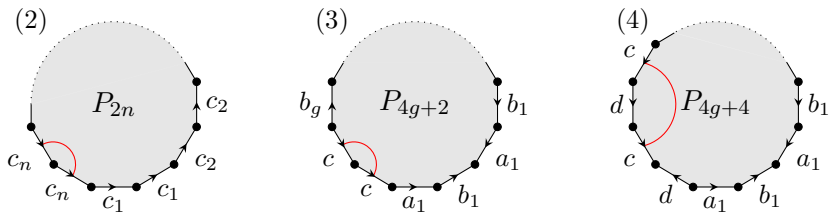

The classification of surfaces asserts that every closed non-orientable surface is homeomorphic to some $\Upsilon_n$, see \cite[Theorem 5.1]{Massey}.  These in turn are distinguished by their Euler characteristics since $\chi(\Upsilon_n) = 2-n$.  The quotient surfaces in cases (2), (3) and (4) of Figure \ref{non-orientable models} have Euler characteristics $2-n$, $1-2g$, and $-2g$, respectively, so they are respectively homeomorphic to $\Upsilon_n$, $\Upsilon_{2g+1}$, and $\Upsilon_{2g+2}$.  The ``$g$'' in each of the latter two cases refers to the fact that the quotient surface is a connected sum of a genus-$g$ orientable surface with a projective plane or Klein bottle, respectively.

The three cases of Figure \ref{non-orientable models} give the following fundamental group presentations:
\begin{align}
  \label{algae} \pi_1(\Upsilon_n) & = \langle\, c_1,\hdots,c_n\,|\, c_1^2c_2^2\cdots c_n^2 = 1\ \rangle \\
  \label{fungi} \pi_1(\Upsilon_{2g+1}) & \cong \langle\, a_1,b_1,\hdots,a_g,b_g,c\,|\, [a_1,b_1]\cdots[a_g,b_g]c^2 = 1\ \rangle \\
  \label{bungee} \pi_1(\Upsilon_{2g+2}) & \cong \langle\, a_1,b_1,\hdots,a_{g},b_{g},c,d\,|\,[a_1,b_1]\cdots[a_{g},b_{g}]cdcd^{-1} = 1\ \rangle\end{align}
The classification of surfaces implies in particular that the latter two presentations yield groups isometric to (even and odd cases, respectively, of) the first.

\begin{lemma}\label{six cases} Let $\gamma$ be a non-nullhomotopic simple closed curve on $\Upsilon_n$.  There is an automorphism $\psi$ of $\Upsilon_n$ such that $\psi(\gamma)$ is represented in $\pi_1\Upsilon_n$ by a member of one of the following classes of elements:\begin{enumerate}
   \item $c_1$, in presentation (\ref{algae});
   \item $c$, in presentation (\ref{fungi});
   \item $a_1$, in presentation (\ref{fungi}) or (\ref{bungee});
   \item $c$, in presentation (\ref{bungee});
   \item $c_1^2\cdots c_{n_0}^2$, for some $n_0<n$, in presentation (\ref{algae}); or
   \item $[a_1,b_1]\cdots[a_{g_0},b_{g_0}]$, for some $g_0\leq g$, in presentation (\ref{fungi}) or (\ref{bungee}).\end{enumerate}
\end{lemma}

\begin{proof}  As it did for Lemma \ref{oogedy boogedy}, here the change of coordinates principle \cite[\S 1.3]{FaMa} implies that a simple closed curve $\gamma$ on $\Upsilon_n$ is determined up to automorphism by the topology of its complement, or equivalently, the complement of a small regular neighborhood of $\gamma$.  (Strictly speaking,  \cite{FaMa} is only concerned with orientable surfaces, but the discussion in \S 1.3.1 there leading up to this assertion does not require it.)  

One source of complication here is that a simple closed curve $\gamma$ on $\Upsilon_n$ may be one- or two-sided, with regular neighborhood homeomorphic to a M\"obius band or annulus, respectively.  In the former case $\gamma$ is not even locally separating, and cutting along $\gamma$ yields a surface with only one boundary component.  We now enumerate possibilities for a non-nullhomotopic such curve $\gamma$:\begin{enumerate}
   \item $\gamma$ is non-separating and one-sided, and its complement is non-orientable.  The curve represented by $c_1$ in presentation (\ref{algae}) has these properties.
   \item $\gamma$ is non-separating and one-sided, with orientable complement.   Note that in this case $n$ must be odd, since the complement in $\Upsilon_n$ of a regular neighborhood of $\gamma$ is an orientable surface with one boundary component.  The curve represented by $c$ in presentation (\ref{fungi}) has these properties.
   \item $\gamma$ is non-separating and two-sided, with non-orientable complement.  The curves represented by $a_1$ in presentations (\ref{fungi}) and (\ref{bungee}) have these properties, in the respective cases $n$ odd and $n$ even.
   \item $\gamma$ is non-separating and two-sided, with orientable complement.  In this case $n$ must be even, for reasons analogous to case (2) above.  The curve represented by $c$ in presentation (\ref{bungee}) has this property.
   \item $\gamma$ is separating, and both components of its complement are non-orientable.  For each $n_0$ with $1\leq n_0<n$, the curve represented by $c_1^2\cdots c_{n_0}^2$ in presentation (\ref{algae}) has this property.
   \item $\gamma$ is separating, and one component of its complement is orientable.  For each $g_0$ with $1\leq g_0\leq g$, the curves represented by $[a_1,b_1]\cdots[a_{g_0},b_{g_0}]$ in presentations (\ref{fungi}) and (\ref{bungee}) have this property.
\end{enumerate}
Since identifying two orientable surfaces along a boundary component always yields an orientable surface, at most one complementary component of $\gamma$ is orientable.  And it follows from the classification of surfaces that for any fixed $k\geq 0$, compact surfaces with $k$ boundary components are classified up to homeomorphism by their orientability and Euler characteristic.  Since cutting along a simple closed curve preserves Euler characteristic, the above list of possibilities is exhaustive.
\end{proof}

If $C_i = \left(\begin{smallmatrix} x_i & y_i \\ 0 & 1/x_i \end{smallmatrix}\right)$, for $x_i\in\mathbb{R}-\{0\}$ and $y_i\in\mathbb{R}$, then $C_i^2 =  \left(\begin{smallmatrix} s_i & t_i \\ 0 & 1/s_i \end{smallmatrix}\right)$ where $s_i=x_i^2$ and $t_i = y_i(x_i+1/x_i)$.  For such matrices $C_1,\hdots,C_n$,
\[ C_1^2\cdots C_n^2 = \begin{pmatrix} s_1\cdots s_n & q_n(s_1,t_1,\hdots,s_n,t_n) \\ 0 & 1/(s_1\cdots s_n) \end{pmatrix}, \]
where
\[ q_n(s_1,t_1,\hdots,s_n,t_n) = s_1\cdots s_{n-1}t_n + \sum_{i=2}^{n-1} \frac{s_1\cdots s_{i-1}t_i}{s_{i+1}\cdots s_n} + \frac{t_1}{s_2\cdots s_n} \]
With the above definitions of $s_i$, $t_i$, and $q_g$ we therefore define:\begin{align*}
	\mathfrak{V}(n) = \{ (x_1,y_1,\hdots,x_n,y_n) & \,|\,x_i\in\mathbb{R}-\{0\}\ \mbox{and}\ y_i\in\mathbb{R}\ \mbox{for all}\ i,\\
	&\quad s_1\cdots s_n = 1,\ \mbox{and}\ q_n(s_1,t_1,\hdots,s_n,t_n) = 0\ \}
\end{align*}
Given the presentation (\ref{algae}), arguing as in Section \ref{spaces} shows that there is a bijective correspondence between $\mathfrak{V}(n)$ and the set of upper-triangular representations of $\pi_1\Upsilon_n$ into $\mathrm{SL}_2(\mathbb{R})$ which takes a $2n$-tuple $(x_1,y_1,\hdots,x_n,y_n)$ to a representation $\rho$ satisfying
\[ \rho(c_1) = C_1 \doteq \left(\begin{smallmatrix} x_1 & y_1 \\ 0 & 1/x_1 \end{smallmatrix}\right),\ \hdots\ ,
\rho(c_n) = C_n \doteq \left(\begin{smallmatrix} x_n & y_n \\ 0 & 1/x_n \end{smallmatrix}\right)  \]
As in Section \ref{spaces}, we regard $\mathfrak{V}(n)$ as a subspace of $M_2(\mathbb{R})^n$ via the embedding $(x_1,\hdots,y_n) \mapsto (C_1,\hdots,C_n)$ as above.  This is again a closed subspace, so it is complete with the subspace metric.

\begin{proof}[Proof of the main theorem]  We need only address the non-orientable case, given Section \ref{orientable proof}.  As there, we handle the with-boundary case by doubling and applying the closed case.  The one exception here is the M\"obius band, i.e.~the one-holed projective plane, whose fundamental group is cyclic and therefore represents faithfully to the upper-triangular matrices in $\mathrm{SL}_2(\mathbb{R})$.  An Euler characteristic calculation shows that the double of each other non-orientable surface with boundary has genus at least four.

We now fix a closed non-orientable surface $\Upsilon_n$ of genus $n\geq 3$.  The proof here parallels the orientable case: for each of the six elements of $\pi_1\Upsilon_n$ listed in Lemma \ref{six cases}, we check that the set of representations that kill it is closed and nowhere dense in $\mathfrak{V}(n)$.  The result then follows as before from Baire's theorem and Lemma \ref{pooka pooka}, whose proof applies without revision to $\mathfrak{V}(n)$.

We will assume below that all $x_i>0$ for each $i$, and more generally that all matrices have positive diagonal entries; that is, we work in $\mathfrak{V}(n)\cap ((0,\infty)\times\mathbb{R})^n$.  This is a closed subspace of $\mathfrak{V}(n)$, since $(0,\infty)$ is closed in $\mathbb{R}-\{0\}$, and it is non-empty since it contains the trivial representation.  As we observed previously, the set of upper-triangular matrices with positive diagonal entries is a torsion-free subgroup of $\mathrm{SL}_2(\mathbb{R})$.

Note also that the map $C\mapsto C^2$ is a self-homeomorphism of this group; or equivalently, that $(x,y)\mapsto (s,t)$, for $s = x^2$ and $t = y(x+1/x)$, is a self-homeomorphism of $(0,\infty)\times\mathbb{R}$.  Given this, below we will deform representations by changing $s$ and $t$, with the understanding that applying the inverse of the above homeomorphism determines a deformation of $x$ and $y$.

The set of representations that kill $c_1$ is closed in $\mathfrak{V}(n)$, being its intersection with $\{(1,0)\}\times\mathbb{R}^{2n-2}$.  Note that also $s_1 =1$ and $t_1 = 0$ for such representations. For any such, $(1,0,s_2,t_2,\hdots,s_n,t_n)\in\mathfrak{V}(n)$, and any $\epsilon>0$, define 
\[ (1,\epsilon,s_2,t_2,\hdots,s_{n-1},t_{n-1},s_n,t_n(\epsilon)) \]
by solving the equation $q_n(s_1,\hdots,t_n(\epsilon)) = 0$ for $t_n(\epsilon)$.  (Note that $q_n$ is linear in $t_i$ for each $i$, and each $t_i$ has a non-zero coefficient.)  Thus the set of representations that kill $c_1$ is nowhere-dense in $\mathfrak{V}(n)$.  This takes care of case (1) of Lemma \ref{six cases}.

We now address case (5) there.  For $n_0<n$, note that the set of representations that kill $c_1^2\cdots c_{n_0}^2$ is the intersection of $\mathfrak{V}(n)$ with the loci $s_1\cdots s_{n_0}=1$ and $q_{n_0}(s_1,\hdots,t_{n_0}) = 0$.  Therefore it is closed in $\mathfrak{V}(n)$.  Let us suppose that $(s_1,t_1,\hdots,s_n,t_n)$ kills $c_1^2\cdots c_{n_0}^2$.  Given $\epsilon>0$, let $s_n(\epsilon) = 1/((s_1+\epsilon)s_2\cdots s_{n-1})$ and define $(s_1+\epsilon,t_1,s_2,\hdots,t_{n-1},s_n(\epsilon),t_n(\epsilon))$ by  solving 
\[ q_n(s_1+\epsilon,t_1,\hdots,t_{n-1},s_n(\epsilon),t_n(\epsilon)) =0\] 
for $t_n(\epsilon)$.  Since $n_0<n$, this defines arbitrarily small deformations of the original one that do not kill $c_1^2\cdots c_{n_0}^2$, finishing case (5).

For the other cases of Lemma \ref{six cases} we will work in representation spaces that are related to our other presentations for $\pi_1(\Upsilon_n)$.  For instance, the set $\mathfrak{V}'(2g+1)$ described below evidently corresponds bijectively to the set of upper-triangular representations from the group presented in (\ref{fungi}) to $\mathrm{SL}_2(\mathbb{R})$.\begin{align*}
	\{ (x_1,y_1,z_1,w_1,&\hdots,x_g,y_g,z_g,w_g,x,y)\,|\,x_i,z_i\in\mathbb{R}-\{0\}\ \mbox{and}\ y_i,w_i\in\mathbb{R}\ \mbox{for all}\ i,\ \\
	&\qquad x,y\in\mathbb{R},\ x^2 =1,\ \mbox{and}\ \sum_{i=1}^g p(x_i,y_i,z_i,w_i) + y(x+1/x) = 0 \} \end{align*}
Here we take $p(x,y,z,w)$ as in Lemma \ref{pee}.

Let $\phi$ be the isomorphism from the group presented in (\ref{fungi}) to $\pi_1(\Upsilon_{n})$ as presented in (\ref{algae}), for $n=2g+1$.  As mentioned above Lemma \ref{six cases}, such an isomorphism exists by the classification of surfaces.  In fact it can be written down directly, but the mere fact that it exists implies, arguing as in Lemma \ref{pooka pooka}, that $\mathfrak{V}'(2g+1)$ is homeomorphic to $\mathfrak{V}(n)$.  To address cases from Lemma \ref{six cases} involving the presentation (\ref{fungi}), we may therefore work with $\mathfrak{V}'(2g+1)$.

As a first observation, we note that since we have assumed $x>0$, $x^2 =1$ if and only if $x=1$.  So the two criteria defining $\mathfrak{V}'(2g+1)$ reduce to $x=1$ and\begin{align}\label{ugga mugga}
	\sum_{i=1}^g p(x_i,y_i,z_i,w_i) + 2y = 0 \end{align}
Now for $g_0\leq g$ and $\bfx=(x_1,\hdots,w_g,1,y)$ specifying a representation killing $[a_1,b_1]\cdots[a_{g_0},b_{g_0}]$ from the presentation (\ref{fungi}) we have $\sum_{i=1}^{g_0} p(x_i,y_i,z_i,w_i)=0$ by the computations of Section \ref{spaces}.  We now argue as in the last few paragraphs of Section \ref{orientable proof}: first apply Lemma \ref{pee} if necessary, perturbing $\bfx$ by an arbitrarily small amount to ensure that $x_i\ne 1$ or $z_i\ne1$ for some $i\leq g_0$ (recall that we are assuming these are positive).  Assuming that, say, $x_i\neq 1$, for $\epsilon>0$ we then take $w_i(\epsilon) = w_i+\frac{\epsilon}{z_i(1-x_i^2)}$, so that\begin{align}\label{poo}
	p(x_i,y_i,z_i,w_i(\epsilon)) + \sum_{j\ne i} p(x_j,y_j,z_j,w_j) = -\epsilon \end{align}
Replacing $y$ by $y+\epsilon$ we thus obtain $(x_1,\hdots,w_i(\epsilon),\hdots,w_g,1,y+\epsilon)\in\mathfrak{V}'(2g+1)$ specifying a representation that does not kill $[a_1,b_1]\cdots[a_{g_0},b_{g_0}]$.  These representations limit to the original as $\epsilon\to 0$, and it follows that the set of representations killing $[a_1,b_1]\cdots[a_{g_0},b_{g_0}]$ is nowhere-dense in $\mathfrak{V}'(2g+1)$.  This handles the subcase of Lemma \ref{six cases}(6) corresponding to presentation (\ref{fungi}).

Inspecting the presentation (\ref{fungi}), we find that a representation of this group that kills $c$ must also kill $[a_1,b_1]\cdots[a_g,b_g]$, so it falls under the $g_0=g$ case of Lemma \ref{six cases}(6) considered above.  The construction above produces representations that do not kill $c$, since $y=0$ initially so $y+\epsilon = \epsilon$, so that construction also takes case of case (2) of Lemma \ref{six cases}.

We finally address the subcase of Lemma \ref{six cases}(3) corresponding to presentation (\ref{fungi}), that is, representations of that group that kill $a_i$.  These correspond to points of $\mathfrak{V}'(2g+1)$ of the form $\bfx=(1,0,z_1,w_1,x_2,\hdots,w_g,1,y)$.  The set of such representations is evidently closed, and we show it is nowhere dense by deforming any such $\bfx$ to $(1,\epsilon,z_1,\hdots,w_g,1,y(\epsilon))$ for
\[ y(\epsilon) = -\frac{1}{2}\left(p(1,\epsilon,z_1,w_1)+\sum_{i=2}^g p(x_i,y_i,z_i,w_i)\right). \]

The cases of Lemma \ref{six cases} that remain to consider are those involving the presentation (\ref{bungee}).  As with the cases involving (\ref{fungi}), we will consider a representation space $\mathfrak{V}'(2g+2)$ which is homeomorphic to $\mathfrak{V}(n)$ for $n=2g+2$, for the same reason as before, but as a subspace of Euclidean space is adapted to the current presentation.  To define $\mathfrak{V}'(2g+2)$ we first observe that for $C = \left(\begin{smallmatrix} x & y \\ 0 & 1/x \end{smallmatrix}\right)$ and $D=\left(\begin{smallmatrix} z & w \\ 0 & 1/z \end{smallmatrix}\right)$ we have\begin{align}\label{goo}
 CDCD^{-1} = \begin{pmatrix} x^2 & xy(z^2+1/x^2)+zw(1-x^2) \\ 0 & 1/x^2 \end{pmatrix} \end{align}
Since the relation of presentation (\ref{bungee}) sets $cdcd^{-1}$ equal to a product of commutators, in $\mathfrak{V}'(2g+2)$ we will necessarily have $x^2 = 1$ in the product above.  When $x^2 = 1$ the upper-right entry of $CDCD^{-1}$ simplifies to $xy(z^2+1)$.  So we define $\mathfrak{V}'(2g+2)$ as\begin{align*}
	 & \{(x_1,y_1,z_1,w_1,\hdots,x_{g+1},y_{g+1},z_{g+1},w_{g+1})\,|\,x_i,z_i\in\mathbb{R}-\{0\}\ \mbox{and}\ y_i,w_i\in\mathbb{R}\ \forall\ i,\\
	&\qquad x_{g+1}^2 = 1,\ \mbox{and}\ \sum_{i=1}^g p(x_i,y_i,z_i,w_i) + x_{g+1}y_{g+1}(z_{g+1}^2+1) = 0 \} \end{align*}
As above, we further restrict to the closed subspace where all $x_i$ and $z_i$ are positive, so in particular we will always take $x_{g+1} = 1$ below.

The subcase of Lemma \ref{six cases}(6) corresponding to presentation (\ref{bungee}) is only a slight modification of the subcase corresponding to (\ref{fungi}): given $\bfx = (x_1,\hdots,w_{g+1})$ specifying a representation killing $[a_1,b_1]\cdots[a_{g_0},b_{g_0}]$ we perturb $\bfx$ as in the previous subcase so that equation (\ref{poo}) again holds, then in this case replace $y_{g+1}$ by $y_{g+1}(\epsilon) = y_{g+1} + \epsilon/(z_{g+1}^2+1)$, yielding $(x_1,\hdots,w_1(\epsilon),\hdots,1,y_{g+1}(\epsilon),z_{g+1},w_{g+1})\in\mathfrak{V}'(2g+2)$ that does not kill $[a_1,b_1]\cdots[a_{g_0},b_{g_0}]$ and converges to $\bfx$ as $\epsilon\to 0$.

Case (4) of Lemma \ref{six cases} also follows from the subcase above, for the same reason that case (2) followed from the subcase of Lemma \ref{six cases}(6) corresponding to presentation (\ref{fungi}).  We are thus left only with the subcase of Lemma \ref{six cases}(3) corresponding to presentation (\ref{bungee}).  But this again follows analogously to the subcase corresponding to presentation (\ref{fungi}).
\end{proof}

\begin{remark}\label{Klein}  The Klein bottle's fundamental group has presentation $\langle c,d\,|\,cdcd^{-1} = 1\rangle$, so for $c\mapsto C$ and $d\mapsto D$ to determine a representation, the product $CDCD^{-1}$ of (\ref{goo}) must equal the identity matrix $I_2$.  This forces $x^2 = 1$, whereupon the upper-right entry simplifies to $\pm y(z^2+1)$.  For this to vanish we must have $y=0$, so $C = \pm I_2$.

If $C= I_2$ then since $c$ represents a simple closed curve the corresponding representation is compressible.  If $C=-I_2$ then the corresponding representation is torsion.  In fact this representation is incompressible if, say, $z=1$ and $w\neq 0$, since the only non-nullhomotopic simple closed curves on the Klein bottle are represented by $c$, $d$, $cd$, $d^2$ and $(cd)^2$.  However, projectivizing it kills $c$.\end{remark}

\section{The non-orientable genus-three surface}

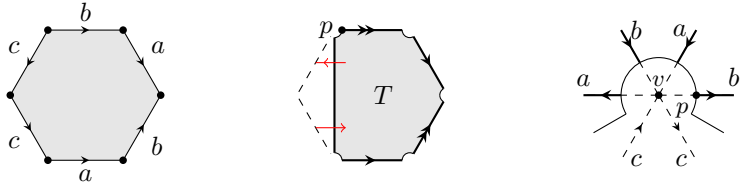
\begin{figure}
\begin{tikzpicture}

\begin{scope}[xshift=-1.5in]

\fill [opacity=0.1] (-0.5,-0.866) -- (0.5,-0.866) -- (1,0) -- (0.5,0.866) -- (-0.5,0.866) -- (-1,0);
\fill (-0.5,-0.866) circle [radius=0.05];
\fill (0.5,-0.866) circle [radius=0.05];
\fill (1,0) circle [radius=0.05];
\fill (0.5,0.866) circle [radius=0.05];
\fill (-0.5,0.866) circle [radius=0.05];
\fill (-1,0) circle [radius=0.05];

\draw [directed] (-0.5,-0.866) to (0.5,-0.866);
\draw [directed] (0.5,-0.866) to (1,0);
\draw [reverse directed] (1,0) to (0.5,0.866);
\draw [reverse directed] (0.5,0.866) to (-0.5,0.866);
\draw [directed] (-0.5,0.866) to (-1,0);
\draw [directed] (-1,0) to (-0.5,-0.866);

\node [below] at (0,-0.866) {$a$};
\node [below right] at (0.75,-0.433) {$b$};
\node [above right] at (0.75,0.433) {$a$};
\node [above] at (0,0.866) {$b$};
\node [above left] at (-0.75,0.433) {$c$};
\node [below left] at (-0.75,-0.433) {$c$};

\end{scope}

\begin{scope}

\fill [opacity=0.1] (-0.5,-0.866) -- (0.5,-0.866) -- (1,0) -- (0.5,0.866) -- (-0.5,0.866);

\draw [thick, directed] (-0.5,-0.866) to (0.5,-0.866);
\draw [thick, ddirected] (0.5,-0.866) to (1,0);
\draw [thick, reverse directed] (1,0) to (0.5,0.866);
\draw [thick, reverse ddirected] (0.5,0.866) to (-0.5,0.866);
\draw [thick] (-0.5,0.866) to (-0.5,-0.866);

\draw [dashed] (-0.5,0.866) -- (-1,0) -- (-0.5,-0.866);

\fill [color=white] (-0.5,-0.866) circle [radius=0.1];
\fill [color=white] (0.5,-0.866) circle [radius=0.1];
\fill [color=white] (1,0) circle [radius=0.1];
\fill [color=white] (0.5,0.866) circle [radius=0.1];
\fill [color=white] (-0.5,0.866) circle [radius=0.1];

\draw (-0.5,-0.766) arc (90:0:0.1);
\draw (0.4,-0.866) arc (180:60:0.1);
\draw (0.9,0) arc (180:120:0.1);
\draw (0.9,0) arc (180:240:0.1);
\draw (-0.5,0.766) arc (-90:0:0.1);
\draw (0.4,0.866) arc (180:300:0.1);

\fill (-0.4,0.866) circle [radius=0.05];
\node [left] at (-0.4,0.866) {$p$};
\node at (0.15,0) {$T$};

\draw [color=red, ->] (-0.35,0.433) to (-0.65,0.433);
\draw [color=red, <-] (-0.35,-0.433) to (-0.75,-0.433);
\draw [color=red] (-0.65,0.433) to (-0.75,0.433);

\end{scope}

\begin{scope}[xshift=1.5in]

\draw [thick, directed] (0.5,0) to (1,0);
\draw [thick, directed] (-0.5,0) to (-1,0);
\draw [thick, reverse directed] (-0.25,0.433) to (-0.5,0.866);
\draw [thick, reverse directed] (0.25,0.433) to (0.5,0.866);

\fill [color=white] (0,0) circle [radius=0.5];
\draw (0.5,0) arc (0:210:0.5);
\draw (0.5,0) arc (0:-30:0.5);
\fill (0.5,0) circle [radius=0.05];
\node [below left] at (0.52,0.02) {\small $p$};
\fill (0,0) circle [radius=0.05];
\node [above] at (0,0) {\small $v$};
\draw (0.433,-0.25) to (0.866,-0.5);
\draw (-0.433,-0.25) to (-0.866,-0.5);

\draw [dashed] (0,0) to (1,0);
\draw [dashed, reverse directed] (0,0) to (-0.5,-0.866);
\draw [dashed, directed] (0,0) to (0.5,-0.866);
\draw [dashed] (0,0) to (-1,0);
\draw [dashed] (0,0) to (0.5,0.866);
\draw [dashed] (0,0) to (-0.5,0.866);

\node [above] at (1,0) {$b$};
\node [left] at (0.5,0.866) {$a$};
\node [right] at (-0.5,0.866) {$b$};
\node [above] at (-1,0) {$a$};
\node [right] at (-0.5,-0.866) {$c$};
\node [left] at (0.5,-0.866) {$c$};

\end{scope}

\end{tikzpicture}
\caption{The non-orientable genus-three surface as a quotient of the hexagon, a punctured torus $T$ inside it, and a neighborhood of the original vertex quotient $v$.}
\label{genus 3 poly}
\end{figure}

In this section we prove Proposition \ref{genus 3}, describing explicit incompressible representations of $\pi_1\Upsilon_3$.  Its presentation used in that result is the $g=1$ case of the presentation (\ref{fungi}).  Correspondingly, see the left side of Figure \ref{genus 3 poly}, which is the $g=1$ case of Figure \ref{non-orientable models}(3).  With this presentation we have the following classification of fundamental group elements representing simple closed curves on $\Upsilon_3$, analogous to Theorem 5.1 of \cite{BS}.

\begin{lemma}\label{more BS} Every word in $\langle a,b,c\,|\,aba^{-1}b^{-1}c^2 = 1\rangle$ that represents a simple, non-nullhomotopic curve on $\Upsilon_3$ is conjugate to one of $a^{\pm 1}$, $b^{\pm 1}$, $c^{\pm 1}$, $c^{\pm 2}$, or to a word of one of the following forms:\begin{enumerate}
	\item\label{torus} up to replacing $a$ with $a^{-1}$, $b$ with $b^{-1}$, or exchanging $a$ and $b$,
	\[ w = a^{n_1}ba^{n_2}b\cdots a^{n_k}b, \]
	where $\{n_1,n_2,\hdots,n_k\}\subset \{n,n+1\}$ for some $n\in\mathbb{N}$; or
	\item $(ac)^{\epsilon}$, $(b^{-1}c)^{\epsilon}$, or $wc^{\epsilon}$ or for a word $w$ as in (\ref{torus}), where $\epsilon\in\{\pm 1,\pm 2\}$.\end{enumerate}
Moreover, for every word $w$ of the form (\ref{torus}) such that $w$, $wc^{\pm1}$, or $(wc^{\pm 1})^2$ represents a simple closed curve on $\Upsilon_3$, $k$ is relatively prime to $l = \sum_{i=1}^k n_i$.
\end{lemma}

\begin{proof}  Inspecting Figure \ref{genus 3 poly}, it is clear that $c$ represents a one-sided simple closed curve $C$ on $\Upsilon_3$.  As we pointed out in the proof of Lemma \ref{six cases}(2), the complement of this curve's regular neighborhood is an orientable surface with one boundary component; in this case a one-holed torus.  Scharlemann proves that $C$ is the unique simple closed curve on $\Upsilon_3$ with orientable complement in Lemma 2.1 of \cite{Scharlemann}.  His discussion on p.~180 there further establishes that every non-nullhomotopic simple closed curve on $\Upsilon_3$ is isotopic to one that either lies in the punctured torus complement of $C$, intersects $C$ once, or bounds a regular neighborhood of a curve intersecting $C$ once.  (All curves that intersect $C$ once are one-sided.)

We will relocate the basepoint from the original vertex quotient $v$ to the point $p$ pictured on the right side of Figure \ref{genus 3 poly}, on the boundary of the regular neighborhood of $C$ (the torus $T$ in the Figure is the complement of such a regular neighborhood).  Let $\alpha$ be the directed arc joining $v$ to $p$ along the dashed line segment in the figure.  It is not hard to see that $\gamma\mapsto \alpha.\gamma.\bar\alpha$ determines an isomorphism from $\pi_1 (\Upsilon_3,p)$ to $\pi_1(\Upsilon_3,v)$ taking the based homotopy classes of $a'$ to $a$, $b'$ to $b$, and $c'$ to $c$, where $a'$ and $b'$ run along the respective arcs of intersection of $a$ and $b$ with $T$ and the curved arc in the boundary of the neighborhood of $v$ pictured on the right side of Figure \ref{genus 3 poly}, and $c' = \bar\alpha.c.\alpha$.  We will therefore abuse notation below and freely refer to the elements $a'$ and $b'$ of $\pi_1 (T,p)$ as $a$ and $b$, and $c'$ as $c$, since the relation $aba^{-1}b^{-1}c^2 = 1$ is still satisfied in this case.

A simple closed curve on $\Upsilon_3$ that is disjoint from $C$ is isotopic into $T$, so by Theorem 5.1 of \cite{BS} it is represented in $\pi_1 (T,p)$ by $a^{\pm 1}$, $b^{\pm 1}$, the commutator $aba^{-1}b^{-1}$ or its inverse, or a word of the form (\ref{torus}).  The relation implies that the commutator equals $c^{-2}$ in $\pi_1 (\Upsilon_3, p)$.  

A simple closed curve that intersects $C$ once can be straightened by an isotopy near $C$ so that its intersection with the complement of $T$ resembles the red arcs in the middle of Figure \ref{genus 3 poly}.  Having done so, let $\gamma_0$ be the arc of intersection with $T$ and $\gamma_1$ the complementary arc, oriented as indicated in the Figure.  

Let $\beta_0$ be the arc from $p$ along $\partial T$ to its near intersection point with $\gamma_0$, the terminal point of $\gamma_0$, in the direction matching the orientation of $c^2$, and let $\beta_1$ be the arc in the same direction to the initial point of $\gamma_0$.  Then $\beta_1.\gamma_0.\gamma_1.\bar\beta_1$ represents the original simple closed curve; $\beta_1.\gamma_0.\bar\beta_0$ represents a simple closed curve on $T$; and $\beta_0.\gamma_1.\bar\beta_1$ has a basepoint-preserving homotopy to $c$.  That is, the original simple closed curve is represented by an element of the form $wc$ for some $w\in\pi_1(T,p)$ representing a simple closed curve on $T$.

If $w = a^{\pm 1}$ or $b^{\pm 1}$ then a regular neighborhood of $w\cup c$ (referring here to their representatives pictured in Figure \ref{genus 3 poly}) is homeomorphic to a twice-punctured projective plane.  One checks directly that $ac$ and $b^{-1}c$ represent simple closed curves, whence it follows from \cite[p.~176]{Scharlemann} that $a^{-1}c$ and $bc$ do not --- the twice-punctured $\mathbb{R}P^2$ has only two isotopy classes of non-boundary parallel, one-sided simple closed curves, and $c$ represents one.  We also observe that since $w=aba^{-1}b^{-1}$ equals $c^2$ in $\pi_1\Upsilon_3$, $wc = c^3$ is not simple and $wc^{-1} = c$ is already listed.  The only remaining possibilities are of the form $wc^{\pm 1}$ for $w$ as in (\ref{torus}).

By Scharlemann's result, all possibilities have now been listed save the squares of the elements described above.  For those elements described above which represent simple closed curves, these curves are one-sided, and so their regular neighborhoods are bounded by simple closed curves represented by the squares of their representatives.

The lemma's final assertion follows from the fact that for any word $w\in\pi_1 T$ that represents a simple closed curve, the inclusion-induced image of $w$ in the fundamental group of the (unpunctured) torus $\overline{T}$ also represents a simple closed curve.  But the inclusion-induced map $\pi_1 T\to \pi_1\overline{T}$ is simply the abelianization, and as is well known, the non-trivial words in $\pi_1\overline{T} = \langle \bar{a},\bar{b}\,|\, [\bar{a},\bar{b}]=1 \rangle$ that represent simple closed curves are of the form $\bar{a}^l\bar{b}^k$ for relatively prime integers $k$ and $l$.  Since every word $w$ of the form (\ref{torus}) has non-trivial abelianization, the result follows.
\end{proof}
 
\begin{prop}\label{genus 3}\GenusThreeProp\end{prop}

\begin{proof}  A straightforward computation checks that the prescribed images of $a$, $b$ and $c$ do determine a representation of $\pi_1\Upsilon_3$, as its relation is satisfied in the target.  (This was also verified in the proof of the main theorem, see in particular equation (\ref{ugga mugga}).)  There is a homomorphism from the set of upper-triangular matrices in $\mathrm{SL}_2(\mathbb{R})$ to the multiplicative group $\mathbb{R}-\{0\}$ that takes a matrix to its upper-left entry.  Composing this map with the representation of $\pi_1\Upsilon_3$ yields one which factors through the abelianization, and it is again straightforward to check that every word listed in Lemma \ref{more BS} maps non-trivially, save powers of $c$, if the hypothesis on $x$ and $z$ are satisfied.  And powers of $c$ are handled by the hypothesis on $p(x,y,z,w)$.

Now suppose we change ``$1$" to ``$-1$'' in the lower-right entry of $C$.  The relation $aba^{-1}b^{-1}c^2=1$ is still satisfied in the target, so this determines a representation to $\mathrm{GL}_2(\mathbb{R})$.  One checks explicitly as above that this kills no simple closed curve and has torsion-free image.  One checks two-sidedness explicitly as well, noting that orientation-reversing curves on $\Upsilon_3$ are represented only by $c$ or words of the form $(ac)^{\pm 1}$, $(b^{-1}c)^{\pm 1}$, and $wc^{\pm 1}$ from Lemma \ref{more BS}(2).  These are exactly the words whose images have negative determinant.\end{proof}

\bibliographystyle{plain}
\bibliography{solvable}

\end{document}